\documentclass[a4paper,12pt]{amsart}
\usepackage{amsmath}
\usepackage{amssymb}
\usepackage{mathrsfs}
\usepackage{enumerate}
\usepackage{ifthen}
\usepackage{graphicx}
\usepackage{subfig}
\usepackage{epstopdf}
\usepackage[T1]{fontenc} 
\usepackage{tabularx}
\usepackage{multirow}
\usepackage{color,soul}
\usepackage{tikz}
\usepackage{pgfplots}
\usepackage{upgreek}

\nonstopmode \numberwithin{equation}{section}
\newtheorem{thm}{Theorem}[section]
\newtheorem{cor}{Corollary}[section]
\newtheorem{lem}{Lemma}[section]

\theoremstyle{definition}

\newtheorem{example}{Example}[section]

\newtheorem{rem}{Remark}[section]

\newcounter{minutes}
\divide\time by 60
\newcounter{hours}
\multiply\time by 60
\addtocounter{minutes}{-\time}

\newcounter {own}
\def\theown {\thesection       .\arabic{own}}

{\qed\bigskip}

\newcounter{alphabet}
\newcounter{tmp}

\begin{document}
\title{ on analytic functions related to Booth-lemniscate}

\author{Md Firoz Ali}
\address{Md Firoz Ali,
 National Institute Of Technology Durgapur, West Bengal, India}
\email{ali.firoz89@gmail.com, fali.math@nitdgp.ac.in}

\author{Md Nurezzaman }
\address{Md Nurezzaman, National Institute Of Technology Durgapur, West Bengal, India}
\email{nurezzaman94@gmail.com}
\author{Lokenath Thakur}
\address{Lokenath Thakur, National Institute Of Technology Durgapur, West Bengal, India}
\email{lokenaththakur1729@gmail.com}


\def\thefootnote{}
\footnotetext{ {\tiny File:~\jobname.tex,
printed: \number\year-\number\month-\number\day,
          \thehours.\ifnum\theminutes<10{0}\fi\theminutes }
} \makeatletter\def\thefootnote{\@arabic\c@footnote}\makeatother

\begin{abstract}
For $0\le \alpha\le 1 $, let $\mathcal{BS}(\alpha)$ be the class of all analytic functions in the unit disk $\mathbb{D}:=\{~z\in\mathbb{C}:|z|<1\}$ with normalization $f(0)=0$ and $f'(0)=1$ 
that satisfy the subordinate relation $zf'(z)/f(z)-1\prec z/(1-\alpha z^2)$ and $\mathcal{BK}(\alpha)$ be the class of all functions $f$ for which $zf' \in \mathcal{BS}(\alpha)$. In this article, we obtain a sharp estimate of the initial Taylor coefficients and logarithmic coefficients for functions in the classes $\mathcal{BS}(\alpha)$ and $\mathcal{BK}(\alpha)$. Further, we obtain the radius of convexity and study the pre-Schwarzian norm for the classes  $\mathcal{BS}(\alpha)$ and $\mathcal{BK}(\alpha)$.

\end{abstract}

\thanks{}

\maketitle
\pagestyle{myheadings}
\markboth{Md Firoz Ali, Md Nurezzaman and Lokenath Thakur}{A study on analytic functions subordinate to Booth-lemniscate}

\section{Introduction}\label{Introduction}
Let $\mathcal{H}$ denote the class of all analytic functions in the unit disk $\mathbb{D}=\{~z\in\mathbb{C}:|z|<1\}$ and $\mathcal{A}$ be the subclass of $\mathcal{H}$ consisting functions $f$ of the form
\begin{align}\label{L-05}
f(z)=z+\sum_{n=2}^\infty a_n z^n.
\end{align}
Further, let $\mathcal{S}$ be the subclass  of $\mathcal{A}$ that contains functions which are  univalent (that is, one-to-one) in $\mathbb{D}$. The univalent function theory  has attracted much more interest for more than a century and is a central area of research in complex analysis. For $f$ and $g$ in $\mathcal{H}$, we say that $f$ is subordinate to $g$, written as $f\prec g$ if there exists a function $\omega$ from $\mathbb{D}$ to itself with $\omega(0)=0$ such that $f(z)=g(\omega(z))$ for $z\in\mathbb{D}$. It is important to note that if $g$ is univalent, then $f\prec g$ if and only if $f(0)=g(0)$ and $f(\mathbb{D})\subset g(\mathbb{D})$. For \( \varphi \in \mathcal{H} \) with $\varphi(0)=1$, let the classes \( \mathcal{S}^*(\varphi) \) and $\mathcal{C}(\varphi)$ be defined by
$$
\mathcal{S}^*(\varphi) = \left\{ f \in \mathcal{A} : \frac{zf'(z)}{f(z)} \prec \varphi(z) \right\},~
\mathrm{and}~~
\mathcal{C}(\varphi) = \left\{ f \in \mathcal{A} : 1 + \frac{zf''(z)}{f'(z)} \prec \varphi(z) \right\},
$$
respectively. The classes $\mathcal{S}^*(\varphi)$ and $\mathcal{C}(\varphi)$ were introduced by Ma and Minda \cite{1992-MaMinda} under the additional hypothesis that \( \varphi \) is univalent with positive real part in \( \mathbb{D} \) such that \( \varphi(\mathbb{D}) \) is symmetric with respect to the real axis and starlike with respect to \( \varphi(0) = 1 \) and \( \varphi'(0) > 0 \). Sometimes, the classes \( \mathcal{S}^*(\varphi) \) and $\mathcal{C}(\varphi)$ are called Ma-Minda classes of starlike and convex functions, respectively. It is important to note that \( f \in \mathcal{C}(\varphi) \) if and only if $zf'(z)\in \mathcal{S}^*(\varphi).$\\

For different choices of the function \( \varphi \), one can get many well known geometric subclasses of \( \mathcal{A} \). For example, if we take \( \varphi(z) = (1+z)/(1-z) \) then the classes \( \mathcal{S}^*(\varphi) \) and $\mathcal{C}(\varphi)$ reduce to the classes \( \mathcal{S}^* \) of starlike functions and \( \mathcal{C} \) of convex functions, respectively. For \( \varphi(z) = (1+(1-2\alpha)z)/(1-z) \), \( 0 \leq \alpha < 1 \), one may get the classes \( \mathcal{S}^*(\alpha) \) of starlike function of order \( \alpha \) and \( \mathcal{C}(\alpha) \) of convex function of order \( \alpha \). For subsequent developments in these classes, see \cite{ 1973-Janowski, 1993-Ronning, 1971-Stankiewicz}. Again, if we take $\varphi(z)=1+F_\alpha(z)$ where $F_\alpha(z)=z/(1-\alpha z^2)$ with $\alpha\in[0,1]$ then the classes \( \mathcal{S}^*(\varphi) \) and \( \mathcal{C}(\varphi) \) reduce to the classes $\mathcal{BS}(\alpha)$ and $\mathcal{BK}(\alpha)$. Thus, a function $f\in \mathcal{A}$ is in $\mathcal{BS}(\alpha)$ if
\begin{align}\label{L-07}
\frac{zf'(z)}{f(z)}-1\prec F_\alpha(z),
\end{align}
and is in the class $\mathcal{BK}(\alpha)$ if 
$$
\frac{zf''(z)}{f'(z)}\prec F_\alpha(z).
$$
The classes $\mathcal{BS}(\alpha)$ and $\mathcal{BK}(\alpha)$ were introduced and studied by Kargar et al.  \cite{2019-Kargar-Ebadian}
and Najmadi et al. \cite{2018-Najmadi}, respectively. In $2013$, Piejko and Sok\'{o}\l~ \cite {2013-Piejko-Sokol} proved that $F_\alpha$ is starlike for $0\le\alpha\le1$ and convex for $0\le\alpha\le3-2\sqrt{2}$.  Note that $F_\alpha(\mathbb{D})= \Omega(\alpha)$, where
\begin{align}\label{L-10}
\Omega(\alpha)=\left\{z:=x+i y\in\mathbb{C}:(x^2+y^2)^2-\frac{x^2}{(1-\alpha)^2}-\frac{y^2}{(1+\alpha)^2} <0\right\},
\end{align}
for $0\le\alpha<1$ and
$$
\Omega(1)=\left\{z:=x+i y\in\mathbb{C}:z\in\mathbb{C}\setminus(-\infty,-\frac{i}{2}]\cup [\frac{i}{2},\infty)\right\}.
$$

\begin{figure}[b]
\subfloat[$\Omega(1/4)$]{\includegraphics[height= 6.5cm, width=4.5cm]{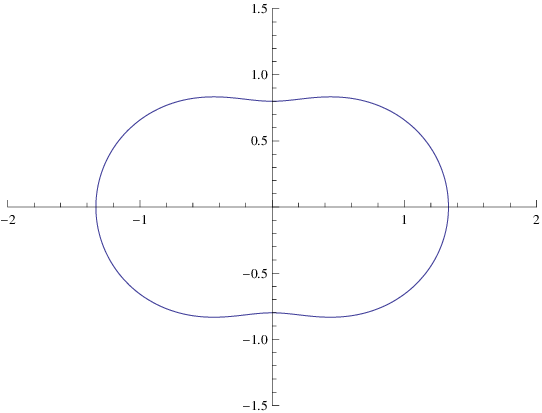}}
\hspace{0.4cm} 
\subfloat[$\Omega(1/2)$]{\includegraphics[height= 6.5cm, width=4.5cm]{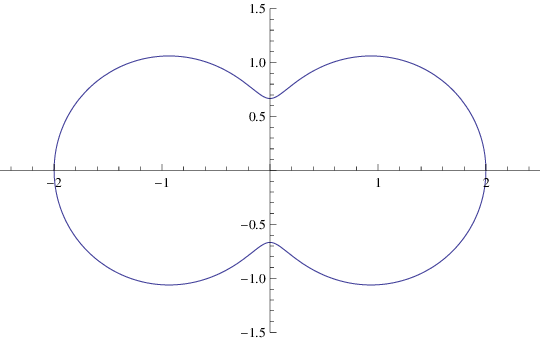}}
\hspace{0.4cm} 
\subfloat[$\Omega(3/4)$]
	{\includegraphics[height= 6.5cm, width=4.5cm]{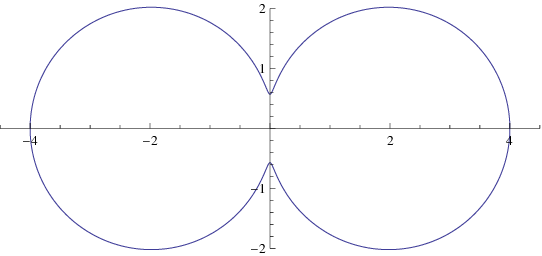}}
	\caption{$\Omega(\alpha)$ for certain values of $\alpha$}
	\label{L-fig-1}
\end{figure}
The images of $\Omega(\alpha)$ for certain values of $\alpha\in[0,1]$ is shown in {\sc Figure \ref{L-fig-1}}. Here we note that a curve $\Gamma$ is called the Booth-lemniscate (named after J. Booth \cite{1873-Booth}) if it is of the form
$$
\Gamma:=\{(x,y)\in\mathbb{R}\setminus\{(0,0)\}:(x^2 +y^2)^2-(n^4+2m^2)x^2-(n^4-2m^2)y^2 =0\}.
$$ 
The curve $\Gamma$ is called elliptic if $n^4 > 2m^2$ while, for $n^4 < 2m^2$, it becomes hyperbolic. One can easily verify that the domain $\Omega(\alpha)$ defined by \eqref{L-10} is bounded by a Booth lemniscate of elliptic type. Let us consider the functions $f_n \in \mathcal{A},~n\in\mathbb{N}$ defined by 
\begin{align}\label{L-02}
f_n(z)=z\left(\frac{1+\sqrt{\alpha}z^n}{1-\sqrt{\alpha}z^n}\right)^{\frac{1}{2n\sqrt{\alpha}}}, ~~0\leq \alpha \leq 1.
\end{align}
Now taking logarithimic derivative on both sides of \eqref{L-02}, one can obtain
$$
\frac{zf_n'(z)}{f_n(z)}-1=\frac{z^n}{1-\alpha z^{2n}} \prec \frac{z}{1-\alpha z^2}=F_\alpha(z). 
$$
This shows that the functions $f_n$ are in the class $\mathcal{BS}(\alpha)$ for each $n\in \mathbb{N}$. An important example of function in the class $\mathcal{BK}(\alpha)$ is given by
$$
g(z)=\int_0^z\exp\left(\int_0^u \frac{dt}{1-\alpha t^2} \right)du=z+\frac{z^2}{2!}+\frac{z^3}{3!}+\frac{1+2\alpha}{4!}z^4+\cdots,\quad z\in\mathbb{D}.
$$

In $1916$, Bieberbach \cite{1916-Bieberbach} proved that if $f\in\mathcal{S}$ is of the form (\ref{L-05}) then $|a_2|\le 2$, and the equality holds if and only if $f$ is a rotation of the Koebe function $k(z)=z/(1-z)^2,~z\in\mathbb{D}$. In the same paper, Bieberbach conjectured that $|a_n|\le n$ for $n\ge3$ when $f\in\mathcal{S}$ and equality holds if and only if $f$ is a rotation of the Koebe function. The logarithmic coefficient $\gamma_n$ of a function $f\in \mathcal{S}$  of the form \eqref{L-05} is defined by
\begin{equation}\label{L-15}
F_f(z)=\log\frac{f(z)}{z}=2\sum\limits_{n=1}^{\infty}\gamma_nz^n.
\end{equation}
The significance of logarithmic coefficients in addressing the coefficient problem for functions in $\mathcal{S}$ was first highlighted by Bazilevi\'c \cite{1965-Bazilevich}. Indeed, Bazilevi\'c \cite{1967-Bazilevich} analyzed the series 
$\sum_{n=1}^{\infty} n|\gamma_n|^2r^{2n}$, which, when multiplied by $\pi$, corresponds to the area of the image of the disk $|z|<r<1$ under the mapping $\frac{1}{2}F_f(z)$ for $f\in\mathcal{S}$. A pivotal result in this domain is the de Branges inequality (resolving the earlier Milin conjecture). For $f\in\mathcal{S}$, it asserts that
$$
\sum_{k=1}^n (n-k+1)|\gamma_n|^2\le\sum_{k=1}^n \frac{n-k+1}{k}, \quad n=1,2,\cdots,
$$
with equality holding precisely for Koebe functions or one of its rotations (see \cite{1985-Branges}). This inequality was instrumental in de Branges proof of the long-standing Bieberbach conjecture. Furthermore, it inspired subsequent inequalities involving logarithmic coefficients, including the bound (see \cite{1979-Duren})
$$
\sum_{k=1}^\infty|\gamma_n|^2\le\sum_{k=1}^\infty \frac{1}{k^2}=\frac{\pi^2}{6}.
$$
While average results for logarithmic coefficients $\gamma_n$ have been studied extensively (see \cite{1979-Duren,1983-Duren,2007-Roth}), precise upper bounds on $|\gamma_n|$ for functions in the class $\mathcal{S}$ remain scarce. The Koebe function $k(z) = z/(1-z)^2$, whose logarithmic coefficients are given by $\gamma_n = 1/n$, typically serves as the extremal function for many problems in $\mathcal{S}$. This might suggest the conjecture that $|\gamma_n| \leq 1/n$ holds generally for functions in $\mathcal{S}$. However, this conjecture fails dramatically not just in precise values but even in terms of asymptotic behavior. 
However, one of the reasons the logarithmic coefficients have received more attention is because although the sharp bound  
\begin{align*}
|\gamma_1|\le 1\quad \mathrm{and}\quad |\gamma_2|\le \frac{1}{2}\left(1+2 e^{-2}\right)
\end{align*}
for the class $\mathcal{S}$ is known,
the problem of determining the sharp bounds of $|\gamma_n|$ for $n\ge3$ in the class $\mathcal{S}$ still remains unsolved.\\


For a locally univalent function $f$, the pre-Schwarzian derivative is defined by
$$P_f(z)=\frac{f''(z)}{f'(z)},$$
and that of the pre-Schwarzian norm (the hyperbolic sup-norm) is defined by
\begin{equation*}\label{P-05}
||P_f||=\sup\limits_{z\in\mathbb{D}}(1-|z|^2)|P_f(z)|. 
\end{equation*}
This norm has a significant meaning in the theory of Teichm\"{u}ller spaces. For a univalent function $f$, it is well known that $||P_f||\leq 6$ and the estimate is sharp. On the other hand, if $||P_f||\leq 1$, then $f$ is univalent in $\mathbb{D}$ (see \cite{1972-Becker, 1984-Becker-Pommerenke}). In 1976, Yamashita \cite{1976-Yamashita} proved that $||P_f ||$ is finite if and only if $f$ is uniformly locally univalent, i.e., a function defined on $\mathbb{D}$ such that for every $z_0 \in \mathbb{D}$, there exists a constant $\delta > 0$ for which the function $f$ is univalent on the hyperbolic disk, 
\(
D_h(z_0, \delta) = \{ z \in \mathbb{D} : \rho_D(z, z_0) < \delta \},
\)
where $\rho_D$ denotes the hyperbolic metric on $\mathbb{D}$. Moreover, if \(||P_f|| < 2\), then \(f\) is bounded in \(\mathbb{D}\) (see \cite{2002-Kim-Sugawa}). In univalent function theory, several researchers obtained sharp estimates of the pre-Schwarzian norm for different subclasses of analytic and univalent functions. Sugawa \cite{1998-Sugawa} obtained sharp estimate of the pre-Schwarzian norm for functions in the class \(\mathcal{S}^*(\varphi)\) with \(\varphi = ((1+z)/(1-z))^{\alpha}\) of strongly starlike functions of order \(\alpha\), \(0 < \alpha \leq 1\). Yamashita \cite{1999-Yamashita} studied the classes \(\mathcal{S}^*(\alpha)\) and \(\mathcal{C}(\alpha)\), \(0 \leq \alpha < 1\), and proved the sharp estimates \(||P_f|| \leq 6 - 4\alpha\) for \(f \in \mathcal{S}^*(\alpha)\) and \(||P_f|| \leq 4(1-\alpha)\) for \(f \in \mathcal{C}(\alpha)\). For more details on the sharp estimates of the pre- Schwarzian norm we refer to \cite{2014-Aghalary, 2023-Ali, 2006-Kim-Sugawa, 2000-Okuyama, 2010-Ponnusamy}.\\

A number $s\in[0,1]$ is called the radius of convexity of a certain subclass $\mathcal{F}$ of $\mathcal{A}$ if $s$ is the largest number such that 
$$
\operatorname{Re}\left(1+\frac{zf''(z)}{f'(z)}\right)>0 \quad \text{for} ~~ |z|<s
$$
and for all $f$ in $\mathcal{F}$. The radius of convexity for the class $\mathcal{S}$ is $r=2-\sqrt{3}$ dating back to the early 20th century, with the foundational work of Nehari \cite{1949-Nehari}. MacGregor \cite{1970-Macgregor} obtained the radius of convexity for the class $\mathcal{S}^*(1/2)$. Todorov \cite{2004-Todorov} obtained the radius of convexity for the class $\mathcal{P}'$ which consists function $f\in\mathcal{A}$ of the form \eqref{L-05} such that $\operatorname{Re}f'(z)>0$. For more details on radius of convexity see \cite{1983-Duren, 1983-Goodman}.\\

In the next section, we study the Taylor coefficient estimate and logarithmic coefficients for functions in the classes $\mathcal{BS}(\alpha)$ and $\mathcal{BK}(\alpha)$. In Section \ref{Radius of convexity}, we study the radius of convexity for classes $\mathcal{BS}(\alpha)$ and $\mathcal{BK}(\alpha)$ and in Section \ref{The pre-Schwarzian norm}, we study the pre-Schwarzian norm for these classes.

\section{coefficient estimates}\label{coefficient estimates}
Let $\mathcal{B}$ be the subclass of $\mathcal{H}$ of function $\omega \in \mathcal{H}$ such that $|\omega(z)|<1$ for all $z\in\mathbb{D}$ and $\mathcal{B}_0$ be the subclass of $\mathcal{B}$ with $\omega(0)=0$. Thus, any $\omega\in\mathcal{B}_0$ has the following Taylor series representation
\begin{align}\label{L-01}
\omega(z)=\sum_{n=1}^\infty c_n z^n.
\end{align}


In $1981$, Prokhorov et al. \cite{1981-Prokhorov} obtained a coefficient inequality for functions in $\mathcal{B}_0$. Here, we present a part of it for our convenience.
\begin{lem}\label{L-52}
Let $f\in\mathcal{B}_0$ be of the form \eqref{L-01}. Then for real $\mu$ and $\nu$, we have
\begin{align*}
|c_3+\mu c_1c_2+\nu c_1^3|\le
\begin{cases}
1,\quad &(\mu, \nu)\in\Omega_1\cup\Omega_2,\\
|\nu|,\quad &(\mu, \nu)\in\Omega_3,
\end{cases}
\end{align*}
where 
\begin{align*}
\Omega_1=&\left\{(\mu, \nu):|\mu|\le\frac{1}{2},\quad -1\le\nu\le1\right\},\\
\Omega_2=&\left\{(\mu, \nu):\frac{1}{2}\le|\mu|\le2,\quad \frac{4}{27}(|\mu|+1)^3-(|\mu|+1)\le\nu\le1\right\},\\
\Omega_3=&\left\{(\mu, \nu):|\mu|\le2,\quad \nu\ge1\right\}.
\end{align*}
Moreover, all inequalities are sharp.
\end{lem}

The next lemma from Keogh et al. \cite{1969-Keogh-Merkes} for the functions in $\mathcal{B}_0$ will be helpful to prove our results.
\begin{lem}\label{L-53}
Let $f\in\mathcal{B}_0$ be of the form \eqref{L-01}. Then for any $\mu \in\mathbb{C}$,
$$
|c_2-\mu c_1^2|\le\max\left\{1,|\mu|\right\}
$$
and the estimate is sharp.
\end{lem}

In our first theorem, we will find the sharp estimate of some initial Taylor coefficients for functions in the class $\mathcal{BS(\alpha)}$.

\begin{thm}\label{L-100}
Let $f\in \mathcal{BS(\alpha)}$, $0\le \alpha\le1$, be of the form \eqref{L-05}. Then
\begin{align*}
|a_2|\le1,\quad |a_3|\le\frac{1}{2!},\\
 |a_4|\le
\begin{cases}
\frac{1}{3},\quad &0\le\alpha\le\frac{1}{2},\\[3mm]
\frac{1+2\alpha}{3!},\quad &\frac{1}{2}\le\alpha\le1.
\end{cases}
\end{align*}
All inequalities are sharp.
\end{thm}

\begin{proof}
If $f\in \mathcal{BS(\alpha)}$ is of the form \eqref{L-05} then
$$
\frac{zf'(z)}{f(z)}-1\prec\frac{z}{1-\alpha z^2}.
$$
Thus, there exist a function $\omega\in\mathcal{B}_0$ of the form \eqref{L-01} such that
$$
\frac{zf'(z)}{f(z)}-1=\frac{\omega(z)}{1-\alpha \omega^2(z)}.
$$
Now equating the coefficients of $z^2,z^3$ and $z^4$, we have
\begin{align}\label{L-85}
a_2=c_1,\quad a_3=\frac{1}{2}(c_1^2+c_2),
\end{align}
\begin{align}\label{L-90}
a_4=\frac{1}{6}\left((1+2\alpha)c_1^3+3c_1c_2+2c_3\right).
\end{align}
From \eqref{L-85} and using Lemma \ref{L-53}, we have
$$
|a_2|=|c_1|\le1 \quad \mathrm{and} \quad |a_3|=\frac{1}{2}|c_1^2+c_2|\le\frac{1}{2!}.
$$
Both the inequalities are sharp for the function $f_1\in \mathcal{BS(\alpha)}$ defined by \eqref{L-02}.
Again, from \eqref{L-90} we have
$$
|a_4|=\frac{1}{3}\left|c_3+\frac{3}{2}c_1c_2+\frac{1+2\alpha}{2}c_1^3\right|=\frac{1}{3}\left|c_3+\mu c_1c_2+\nu c_1^3\right|,
$$
where $\mu=3/2$ and $\nu=\frac{1+2\alpha}{2}$. Now we consider the following two cases.\\

\textbf{Case-1:} Let $0\le\alpha\le\frac{1}{2}$. 
Using the notation of Lemma \ref{L-52}, we have $\mu=3/2\le2$ and $\nu=(1+2\alpha)/2\le1$ that show that $(\mu,\nu)\in \Omega_2 $. Therefore, by Lemma \ref{L-52},
$$
|a_4|\le\frac{1}{3}.
$$
The inequality is sharp for the function $f_3\in \mathcal{BS(\alpha)}$ defined by \eqref{L-02}.\\

\textbf{Case-2:} Let $\frac{1}{2}<\alpha\le1$. 
Using the notation of Lemma \ref{L-52}, we have $\mu=3/2\le2$ and $\nu=(1+2\alpha)/2\ge1$ which show that $(\mu,\nu)\in \Omega_3 $. Therefore, by Lemma \ref{L-52},
$$
|a_4|\le\frac{1+2\alpha}{3!}
$$
and the inequality is sharp for the function $f_1\in \mathcal{BS(\alpha)}$ defined by \eqref{L-02}.
\end{proof}

The next result gives the sharp estimate of $|a_2|$, $|a_3|$ and $|a_4|$ for functions in the class $\mathcal{BK(\alpha)}$ and it immediately follows from Theorem \ref{L-100}. Here we note that the estimate of $|a_2|$ and $|a_3|$ is obtained by Najmadi et al. \cite{2018-Najmadi}.

\begin{cor}\label{L-101}
Let $f(z)\in \mathcal{BK}(\alpha)$, $0\le \alpha\le1$, be of the form \eqref{L-05}. Then
\begin{align*}
|a_2|\le \frac{1}{2}, \quad |a_3|\le \frac{1}{6},\\ |a_4|\le \begin{cases}
\frac{1}{12},\quad &0\le\alpha\le\frac{1}{2},\\[3mm]
\frac{1+2\alpha}{24},\quad &\frac{1}{2}\le\alpha\le1.
\end{cases}
\end{align*}
The estimates $|a_2|\le\frac{1}{2}$,  $|a_3|\le\frac{1}{6}$ and  $|a_4|\le\frac{1+2\alpha}{24}$ are sharp for the function $g_1\in \mathcal{BK(\alpha)}$ defined by
$$
g_1(z)=\int_0^z\exp\left(\int_0^u \frac{dt}{1-\alpha t^2} \right)du=z+\frac{z^2}{2!}+\frac{z^3}{3!}+\frac{1+2\alpha}{4!}+\cdots.
$$
The estimate $|a_4|\le\frac{1}{12}$ is sharp for the function $g_2\in \mathcal{BK(\alpha)}$ defined by
$$
g_2(z)=\int_0^z\left(\frac{1+\sqrt{\alpha}t^3}{1-\sqrt{\alpha}t^3}\right)^{\frac{1}{6\sqrt{\alpha}}}dt=z+\frac{1}{12}z^4+\cdots.
$$
\end{cor}


In the next theorem, we obtain sharp estimates of the logarithmic coefficients for functions in the class $\mathcal{BS(\alpha)}$.
\begin{thm}
  Let $f\in \mathcal{BS(\alpha)}$ be of the form \eqref{L-05} where $0\le \alpha \le 1$ and the logarithmic coefficients $\gamma_{n}$ of $f(z)$ be given by \eqref{L-15}. For $0\le\alpha\le1$ and $n=1,2,3$, the logrithmic coefficient $\gamma_n$ satisfies the sharp inequality 
  \begin{align}\label{L-95}
  | \gamma_{n}| \leq \frac{1}{2n}.
  \end{align}
  For $0\le \alpha \le 3-2\sqrt{2},$ the inequality \eqref{L-95} holds for all $n\in\mathbb{N}$ and is sharp.
  Moreover, for $0\le\alpha\le1$ and every $n\in\mathbb{N}$, we have 
  \begin{align}\label{L-94}
  | \gamma_{n}| \leq \frac{1}{2}.
  \end{align}
 \end{thm}

 \begin{proof}
 If $f\in \mathcal{BS(\alpha)}$ is of the form \eqref{L-05}, then 
\begin{align*}
\frac{zf'(z)}{f(z)}-1\prec\frac{z}{1-\alpha z^2}=F_\alpha(z).
\end{align*}
Equivalently,
\begin{align}\label{L-97}
\sum_{n=1}^{\infty}2n\gamma_{n}z^n=z\frac{d}{dz}\left(\log\left(\frac{f(z)}{z}\right)\right) = \frac{zf'(z)}{f(z)}-1\prec
 F_\alpha(z).
\end{align}
Thus, there exist $\omega\in\mathcal{B}_0$ of the form \eqref{L-01} such that
$$
\sum_{n=1}^{\infty}2n\gamma_{n}z^n=\frac{\omega(z)}{1-\alpha (\omega(z))^2}=c_1z+c_2z^2+(\alpha c_1^3+c_3)z^3+\cdots.
$$
Equating coefficients of $z^n$ for $n=1,2,3$ we have
$$
\gamma_{1}=\frac{c_1}{2},\quad \gamma_{2}=\frac{c_2}{4},\quad \gamma_{3}=\frac{\alpha c_1^3+c_3}{6}.
$$
Therefore
$$
|\gamma_{1}|=\frac{|c_1|}{2}\le\frac{1}{2},
$$
and the inequality is sharp for the function $f_1\in\mathcal{BS}(\alpha)$ defined by \eqref{L-02}. Again, 
$$
 |\gamma_{2}|=\frac{|c_2|}{4}\le\frac{1}{4},
$$
and the inequality is sharp for the function $f_2\in\mathcal{BS}(\alpha)$ defined by \eqref{L-02}. Further,
$$
|\gamma_{3}|=\frac{|c_3+\alpha c_1^3|}{6}=\frac{|c_3+\mu  c_1c_2+\nu c_1^3|}{6},
$$
where $\mu=0$ and $\nu=\alpha$. Using the notation of Lemma \ref{L-52} we can see that $(\mu,\nu)\in\Omega_1$. Therefore Lemma \ref{L-52}, gives 
$$
|\gamma_{3}|\le\frac{1}{6},
$$
and the inequality is sharp for the function $f_3\in\mathcal{BS}(\alpha)$ defined by \eqref{L-02}.\\

Further, $F_\alpha$ is convex for $ 0\leq \alpha \leq 3-2\sqrt{2}$. Therefore from subordination relation \eqref{L-97} and by Rogosinski's theorem \cite{1943-Rogosinski}, we obtain, 
$$
2n|\gamma_n|\le 1 \implies|\gamma_{n}| \leq \frac{1}{2n}, \quad \mathrm{for ~all}~n\in \mathbb{N}.
$$

To show that the inequality \eqref{L-95} is sharp, let us consider the function $f_n$ defined by \eqref{L-02}. For the function $f_n(z)$, we have
 $$
 \log{\frac{f_n(z)}{z}}=\frac{1}{n}z^n+\cdots.
 $$  
Note that, $F_\alpha$ is starlike for $ 0\leq \alpha <1$. Then from subordination relation \eqref{L-97} and by Rogosinski's  theorem \cite{1943-Rogosinski}, we obtain, 
$$
2n|\gamma_n|\le n \implies|\gamma_{n}| \leq \frac{1}{2}, \quad \mathrm{for ~all}~n\in \mathbb{N}.
$$
\end{proof}

%

\section{Radius of convexity}\label{Radius of convexity}
In $2018$, Cho et al. \cite{2018-Cho} studied the radius of convexity for the class $\mathcal{BS(\alpha)}$ and conjectured that the radius of convexity of the class $\mathcal{BS(\alpha)}$ is $r=\min\{r_\alpha,r'\}$ where $r_\alpha$ is the root of $1-r-\alpha r^2=0$ and $r'$ is the smallest positive root of $1-3r+(1-6\alpha)r^2+5\alpha r^3 + 5 \alpha^2 r^4=0$. In the next theorem, we prove that the above conjecture is not true and we determine the exact radius of convexity for the class $\mathcal{BS(\alpha)}$.

\begin{thm}\label{L-54}
The radius of convexity for the class $\mathcal{BS(\alpha)}$ is $r_\alpha=\min\{r'_\alpha,r''_\alpha\}$
where $r'_\alpha$ and $r''_\alpha$ are the roots of 
$$
l_\alpha(r):=\alpha^2r^4+\alpha r^3+(1-2\alpha)r^2-3r+1=0 \quad and ~~ m_\alpha(r):=1-r-\alpha r^2=0
$$
respectively, lies in $(0,1)$.
\end{thm}

\begin{proof}
If $f\in \mathcal{BS(\alpha)}$, then from \eqref{L-07} we have
$$
\frac{zf'(z)}{f(z)}-1\prec\frac{z}{1-\alpha z^2}.
$$
Thus, there exist $\omega\in\mathcal{B}_0$ such that
\begin{align}\label{L-56}
\frac{zf'(z)}{f(z)}-1=\frac{\omega(z)}{1-\alpha \omega^2(z)}.
\end{align}
Note that any functions $\omega$ in $\mathcal{B}_0$ satisfies 
\begin{align}\label{L-60}
|\omega(z)|\le|z| \quad \text{and} ~~
|\omega'(z)|\le\frac{1-|\omega(z)|^2}{1-|z|^2}.
\end{align}
Taking logarithmic derivatives on both sides of \eqref{L-56} and by a simple calculation we get
\begin{align*}
1+\frac{zf''(z)}{f'(z)}=1+\frac{\omega(z)}{1-\alpha \omega^2(z)}+\left(\frac{1+\alpha\omega^2(z)}{(1+\omega(z)-\alpha \omega^2(z))(1-\alpha \omega^2(z))}\right)z\omega'(z).
\end{align*}
Therefore, using relation \eqref{L-60} and by a simple calculation
\begin{align}\label{L-65}
&\operatorname{Re}\left(1+\frac{zf''(z)}{f'(z)}\right)\\\quad \quad&\ge1-\frac{|\omega(z)|}{1-\alpha |\omega^2(z)|}-\frac{|z|(1+\alpha|\omega^2(z)|)(1-|\omega(z)|^2)}{(1-|z|^2)(1-|\omega(z)|-\alpha |\omega^2(z)|)(1-\alpha |\omega^2(z)|)}\nonumber\\&=1-\frac{s}{1-\alpha s^2}-\frac{r(1+\alpha s^2)(1-s^2)}{(1-r^2)(1-s-\alpha s^2)(1-\alpha s^2)}=:h(r,s)\nonumber,
\end{align}
where $0\le s=|\omega(z)|\le r=|z|<1$. Now, 
\begin{align}\label{L-70}
\frac{\partial}{\partial s}h(r,s)=-\frac{1
+\alpha s^2}{(1-\alpha s^2)^2}-\frac{r \psi_\alpha(s)}{(1-r^2)(1-s-\alpha s^2)^2(1-\alpha s^2)^2},
\end{align}
where 
$$
\psi_\alpha(s)=1-(2-6\alpha)s+(1-4\alpha)s^2-4\alpha(1+\alpha)s^3+\alpha(4-\alpha)s^4+2\alpha^2(3-\alpha)s^5-\alpha^2 s^6.
$$
A simple but tedious arrangement gives
\begin{align*}
\psi_\alpha(s)&=
1-2s-s^2-2\alpha^3 s^5+\alpha(6s-4s^2-4s^3+4s^4)-\alpha^2(4s^3+s^4-6s^5+s^6)\\&=(1-s)^2+2\alpha s^4(1-\alpha^2 s)+\alpha(6s-4s^2-4s^3+2s^4)-(1-s)^2\alpha^2s^4\\&\quad-4\alpha^2s^3+4\alpha^2 s^5\\&=(1-s)^2(1-\alpha^2 s^4)+2\alpha s^4(1-\alpha^2 s)+4\alpha s(1-s)+2\alpha s (1-s^2)\\&\quad-2\alpha s^3(1-s)-4\alpha^2s^3(1-s^2)\\&=(1-s)^2(1-\alpha^2 s^4)+2\alpha s^4(1-\alpha^2 s)-2\alpha^2s^4(1-s)+2\alpha s(1-s)(1-s^2)\\&\quad+2\alpha s (1-s^2)(1-\alpha s^2)+2\alpha s (1-s)(1-\alpha s^2)
\\&=(1-s)^2(1-\alpha^2 s^4)+2\alpha s(1-s)(2-s^2-\alpha s^2)+2\alpha s(1-s^2)(1-\alpha s^2)\\&\quad+2\alpha s^4(1-\alpha)(1+\alpha s)
\\&\ge 0~~\text{for}~~ \alpha\in [0,1]~~ \text{and} ~~s\in[0,1).
\end{align*}
 Hence, from \eqref{L-70} one can easily verify that
$$
\frac{\partial}{\partial s}h(r,s)<0,\quad 0\le s\le r<1.
$$
Thus, $h$ is a decreasing function in $s$ for $0\le s\le r<1$. Therefore, from \eqref{L-65} we have
\begin{align}\label{L-75}
\operatorname{Re}\left(1+\frac{zf''(z)}{f'(z)}\right)\ge  h(r,r)=1-\frac{r(2-r)}{(1-\alpha r^2)m_\alpha(r)}=\frac{l_\alpha(r)}{(1-\alpha r^2)m_\alpha(r)}\ge h(r,s),
\end{align}
where $l_\alpha(r)$ and $m_\alpha(r)$ are defined above. Now, we consider the following two cases.\\

\textbf{Case-1:} Let $0<\alpha\le1$. Then $l_\alpha(0)=1>0$ and $l_\alpha(1)=\alpha(\alpha-1)-1<0$. So $l_\alpha$ has either one or three real roots between $0$ and $1$. If possible, let $l_\alpha$ has three real roots in $(0,1)$. Then by Rolle's theorem $l'_\alpha$ (which is a polynomial of degree 3) must have two real roots between $0$ and $1$. But, by a simple calculation one can easily verify that the discriminant $\Delta(l'_\alpha)$ is 
$$
\Delta(l'_\alpha)=4\alpha^2[-23-87\alpha-416\alpha^2-256\alpha^2(1-\alpha)]<0.
$$
So $l'_\alpha$ has two complex roots. Thus, $l_\alpha$ has a unique real root, say $r'_\alpha$, in $(0,1)$. On the other hand, we have $m_\alpha(0)=1>0$, $m_\alpha(1)=-\alpha<0$ and also $m'_\alpha(r)=-1-2\alpha r<0$. Thus, $m_\alpha(r)$ has a unique real root $r''_\alpha=\left(-1+\sqrt{1+4\alpha}\right)/2\alpha$ in $(0,1)$. Let, $r_\alpha=\min\{r'_\alpha,r''_\alpha\}$.\\

\textbf{Case-2:} Let $\alpha=0$. Then $l_0(r)=0$ for $r=r'_\alpha=(3-\sqrt{5})/2$ and $m_0(r)=0$ implies $r=r''_\alpha=1$. Let, $r_\alpha=\min\{r'_\alpha,r''_\alpha\}=(3-\sqrt{5})/2$. \\

Combining \textbf{Case 1} and \textbf{Case 2}, the inequality \eqref{L-75} gives
$$
\operatorname{Re}\left(1+\frac{zf''(z)}{f'(z)}\right)>0 \quad \mathrm{for} \quad |z|<r_\alpha.
$$
Now, let us consider the function $f_1\in \mathcal{BS(\alpha)}$ defined by \eqref{L-02}.
Then,
$$
1+\frac{zf_1''(z)}{f_1'(z)}=1+\frac{z}{1-\alpha z^2}+\frac{z(1+\alpha z^2)}{(1+z-\alpha z^2)(1-\alpha z^2)},
$$ 
and so
$$
\operatorname{Re}\left(1+\frac{-rf_1''(-r)}{f_1'(-r)}\right)=\frac{l_\alpha(r)}{m_\alpha(r)(1-\alpha r^2)}
\begin{cases}
=0,\quad \mathrm{for} &r=r'_\alpha,\\
=\infty,\quad \mathrm{for} &r=r''_\alpha, \\
\le 0,\quad \mathrm{for}& r>r'_\alpha.
\end{cases}
$$
This shows that the radius of convexity of $\mathcal{BS}(\alpha)$ is $r_\alpha=\min\{r'_\alpha,r''_\alpha\}$.
\end{proof}

In the next theorem, we determine the radius of convexity for 
the class $\mathcal{BK}(\alpha)$.

\begin{thm}
The radius of convexity of the class $\mathcal{BK}(\alpha)$ is
$$
r_\alpha=\frac{\sqrt{1+4\alpha}-1}{2\alpha},\quad \alpha\in[0,1].
$$
\end{thm}

\begin{proof}
Let $f\in\mathcal{BK(\alpha)}$ be of the form \eqref{L-05}. Then
$$
\frac{zf''(z)}{f'(z)}\prec \frac{z}{1-\alpha z^2}.
$$
Thus, there exists a function $\omega\in\mathcal{B}_0$ such that
$$
\frac{zf''(z)}{f'(z)}=\frac{\omega(z)}{1-\alpha \omega^2(z)}.
$$
Therefore, for $0\le|z|=r<1$
\begin{align*}
\operatorname{Re}\left(1+\frac{zf''(z)}{f'(z)}\right)=&\operatorname{Re}\left(1+\frac{\omega(z)}{1-\alpha \omega^2(z)}\right)\\\ge& 1-\frac{|\omega(z)|}{|1-\alpha \omega^2(z)|}\\\ge&1-\frac{|z|}{1-\alpha |z|^2}\\=&\frac{1-r-\alpha r^2}{1-\alpha r^2}>0,~~\text{for}~~r<r_\alpha:=\frac{\sqrt{1+4\alpha}-1)}{2\alpha}.
\end{align*}

Now, let $g_1\in\mathcal{BK(\alpha)}$ be defined by
$$
g_1(z)=\int_0^z\exp\left(\int_0^u \frac{dt}{1-\alpha t^2} \right)du\quad z\in\mathbb{D}.
$$
Therefore,
\begin{align*}
\operatorname{Re}\left(1+\frac{zg_1''(z)}{g_1'(z)}\right)=\operatorname{Re}\left(1+\frac{z}{1-\alpha z^2}\right),
\end{align*}
and so
\begin{align*}
\operatorname{Re}\left(1+\frac{-rg_1''(-r)}{g_1'(-r)}\right)=\operatorname{Re}\left(\frac{1-r-\alpha r^2}{1-\alpha r^2}\right)=0,\quad \mathrm{for}~r=r_\alpha.
\end{align*}
This shows that the radius of convexity of $\mathcal{BK}(\alpha)$ is $r_\alpha$.
\end{proof}

\section{The pre-Schwarzian norm}\label{The pre-Schwarzian norm}
In the next theorem, we determine sharp estimate of the pre-Schwarzian norm for functions in the class $\mathcal{BK}(\alpha)$.

\begin{thm}\label{L-105}
If $f\in\mathcal{BK(\alpha)}$, then $\|P_f\|\le1$ and the estimate is sharp.
\end{thm}

\begin{proof}
If $f\in\mathcal{BK(\alpha)}$, then
$$
\frac{zf''(z)}{f'(z)}\prec \frac{z}{1-\alpha z^2}.
$$
Thus, there exist $\omega\in\mathcal{B}_0$ of the form \eqref{L-01} such that
$$
\frac{zf''(z)}{f'(z)}=\frac{\omega(z)}{1-\alpha \omega^2(z)}.
$$
By using Schwarz lemma, we get
$$
\|P_f\|=\sup_{z\in\mathbb{D}}(1-|z|^2)\left|\frac{f''(z)}{f'(z)}\right|=\sup_{z\in\mathbb{D}}(1-|z|^2)\frac{|\omega(z)|}{|z(1-\alpha \omega^2(z))|}\le\sup_{0\le|z|<1}\frac{1-|z|^2}{1-\alpha |z|^2}\le1.
$$
To show that the estimate is sharp, let $g\in\mathcal{BK(\alpha)}$ be defined by
$$
\frac{zg''(z)}{g'(z)}=\frac{z}{1-\alpha z^2}.
$$
Therefore,
$$
1\ge\|P_g\|=\sup_{z\in\mathbb{D}}(1-|z|^2)\left|\frac{g''(z)}{g'(z)}\right|=\sup_{0\le|z|<1}\frac{1-|z|^2}{|1-\alpha z^2|}\ge \sup_{0\le r<1}\frac{1-r^2}{1-\alpha r^2}=1 
$$
and so $\|P_g\|=1$.
\end{proof}

\begin{rem}
Since the pre-Schwarzian norm $\|P_f\|\le1$ for functions in $\mathcal{BK(\alpha)}$ 
it follows that functions in $\mathcal{BK(\alpha)}$ are also in $\mathcal{S}$ (see \cite{1972-Becker, 1984-Becker-Pommerenke}).  
\end{rem}

In the next example, we provide a function in the class $\mathcal{BS}(\alpha)$ with unbounded pre-Schwarzian norm.

\begin{example}\label{L-130}
Let us consider the function $f_1\in \mathcal{BS(\alpha)}$ of the form
$$
f_1(z)=z\left(\frac{1+\sqrt{\alpha}z}{1-\sqrt{\alpha}z}\right)^{\frac{1}{2\sqrt{\alpha}}}.
$$
Then,
\begin{align*}
\frac{f_1''(z)}{f_1'(z)}=&\frac{1}{1-\alpha z^2}+\frac{1+\alpha z^2}{(1+z-\alpha z^2)(1-\alpha z^2)}\\=&\frac{2+z}{(1+z-\alpha z^2)(1-\alpha z^2)}
\end{align*}
and so
\begin{align*}
\|P_{f_1}\|=&\sup_{z\in\mathbb{D}}(1-|z|^2)\left|\frac{f_1''(z)}{f_1'(z)}\right|\\
=&\sup_{z\in\mathbb{D}}(1-|z|^2)\frac{|2+z|}{|(1+z-\alpha z^2)(1-\alpha z^2)|}\ge \sup_{0\le r<1}\frac{(1-r^2)(2-r)}{|g_\alpha(r)|(1-\alpha r^2)},
\end{align*}
where $g_\alpha(r)=1-r-\alpha r^2$. Since $g_\alpha$ has a root $(\sqrt{1+4\alpha}-1)/2\alpha$ in $[0,1)$, it follows that
$$
\|P_{f_1}\|=\infty.
$$
\end{example}

\begin{rem}
The pre-Schwarzian norm for the function f in Example \ref{L-130} is infinite. This shows that functions in the class $\mathcal{BS(\alpha)}$ are not necessarily univalent (see \cite{1976-Yamashita}).
\end{rem}

In the next example, we provide a function in the class $\mathcal{BS}(\alpha)$ with bounded pre-Schwarzian norm.

\begin{example}
Let us consider the function  $f\in \mathcal{BS(\alpha)}$ defined by $f(z)=z.$ Then,
\begin{align*}
\|P_{f}\|=&\sup_{z\in\mathbb{D}}(1-|z|^2)\left|\frac{f''(z)}{f'(z)}\right|=0.
\end{align*}
\end{example}

\vspace{4mm}
\noindent\textbf{Data availability:}
Data sharing is not applicable to this article as no data sets were generated or analyzed during the current study.\vspace{4mm}
\\
\noindent\textbf{Authors Contributions:}
All authors contributed equally to the investigation of the problem and the order of the authors is given alphabetically according to their surname. All authors read and approved the final manuscript.\vspace{4mm}
\\
\noindent\textbf{Acknowledgement:}
The third author thanks the CSIR for the financial support through CSIR-SRF Fellowship ( File no.: 09/0973(13731)/2022-EMR-I ).


\end{document}